\documentclass[a4paper, oneside, 10pt, reqno]{amsproc}
\usepackage{mathrsfs}
\usepackage{url}
\usepackage{amsthm}
\usepackage{amsfonts}
\usepackage{amsmath}
\usepackage{amssymb}
\usepackage{mathrsfs}
\usepackage{epsfig}
\usepackage{graphicx}
\usepackage{epstopdf}
\usepackage{color}
\usepackage{ifthen}
\usepackage{float}
\usepackage[active]{srcltx}
\usepackage{marginnote}
\makeatletter
\@namedef{subjclassname@2020}{%
	\textup{2020} Mathematics Subject Classification}
\makeatother

\newtheorem{theorem}{Theorem}[section]

\newtheorem{proposition}{Proposition}[section]
\newtheorem{lemma}{Lemma}[section]
\newtheorem{remark}{Remark}[section]

\numberwithin{equation}{section}
\newcommand{\beq}{\begin{eqnarray}}
	\newcommand{\eeq}{\end{eqnarray}}
\newcommand{\beqn}{\begin{eqnarray*}}
	\newcommand{\eeqn}{\end{eqnarray*}}
\newcommand{\rar}{\rightarrow}
\newcommand*{\Ge}{\geqslant}
\newcommand*{\Le}{\leqslant}
\newcommand*{\inp}[2]{\langle{#1},\,{#2} \rangle}
\newcommand*{\mf}{\mathbf}

\newcommand{\bc}{\begin{centre}}
	\newcommand{\ec}{\end{centre}}

\newcommand{\ba}{\begin{array}}
	\newcommand{\ea}{\end{array}}

\renewcommand{\Delta}{{\nabla}}

\begin{document}
\title[Nevanlinna-Pick interpolation in the right half-plane]
{Nevanlinna-Pick interpolation \\ in the right half-plane}

\author{Sameer Chavan (Corresponding author)}
\author{Chaman Kumar Sahu}
\address{Department of Mathematics and Statistics\\
	Indian Institute of Technology Kanpur, India}
\email{chavan@iitk.ac.in}
\address{Department of Mathematics\\
	Indian Institute of Technology Mumbai, India}
\email{chamanks@math.iitb.ac.in, sahuchaman9@gmail.com}



\begin{abstract}
The Szegö-Dirichlet kernel of the right half-plane $\mathbb H_{1/2}$ is given by 
${\varkappa}(s, u) = \zeta(s+\overline{u}),$ $s, u \in \mathbb H_{1/2},$
where $\zeta$ denotes the Riemann zeta function. We show that none of the positive integer powers of $\varkappa$ has the $2$-point scalar Pick property. Nevertheless, a network realization formula for the right half-plane $\mathbb H_0$ is obtained.   
\end{abstract}

\maketitle

\section{Introduction and main results}

For a nonempty open connected subset $\Omega$ in $\mathbb C,$ 
let $H^{\infty}(\Omega)$ denote the space of scalar-valued bounded holomorphic functions on $\Omega.$ The closed unit ball of $H^\infty(\Omega)$ is denoted by $H_1^\infty(\Omega).$ 
Given a positive semi-definite kernel $\kappa : \Omega \times \Omega \rar \mathbb C,$ by Moore's theorem (see \cite[Theorem~2.5]{AMY}), there exists a reproducing kernel Hilbert space of scalar-valued functions on $\Omega$ associated with $\kappa.$ Recall that the function $\kappa(\cdot, u) \in \mathscr H_\kappa$ and $\inp{\kappa(\cdot, u)}{\kappa(\cdot, s)}=\kappa(s, u)$ for every $s, u \in \Omega$ (refer to \cite{AMY, P-R} for basics of reproducing kernel Hilbert spaces). 
Sometimes, we use the symbol $\kappa_{u}$ for the  function $\kappa(\cdot, u).$ 
For a positive integer $m,$ the $m$-fold pointwise product of $\kappa$ with itself is denoted by $\kappa^m.$ The multiplier algebra of $\mathscr H_\kappa$ will be denoted by $\mathcal M(\mathscr H_\kappa).$ If $A$ is a positive semi-definite matrix with complex entries, then we denote this by $A \succeq 0.$ If $T$ is a bijective bounded linear transformation between two Hilbert spaces, then $T^{-1}$ denotes the bounded inverse of $T.$

The Nevanlinna-Pick interpolation problem in $\Omega$ asks, given distinct points $\lambda_1, \ldots, \lambda_n \in \Omega$ and $w_1, \ldots, w_n \in \mathbb C$, whether there exists $\varphi \in \mathcal M(\mathscr H_\kappa)$ of norm at most $1$ such that
\beq \label{inter-new}
\varphi(\lambda_j)=w_j, \quad j=1, \ldots, n.
\eeq
The positive semi-definiteness of the Pick matrix $
 [(1-w_j\overline{w}_k)\kappa(\lambda_j, \lambda_k)]_{j, k=1}^n$ is a necessary condition to solve the above interpolation problem (see \cite[Theorem 5.2]{A-M}).  
Following \cite{A-M}, we say that the kernel $\kappa$ on $\Omega$ has {\it the $n$-point scalar Pick property} if, for distinct points $\lambda_1, \ldots, \lambda_n \in \Omega$ and $w_1, \ldots,w_n \in \mathbb D$, $$
 [(1-w_j\overline{w}_k)\kappa(\lambda_j, \lambda_k)]_{j, k=1}^n \succeq 0,$$ 
	then there exists a multiplier $\varphi \in \mathcal M(\mathscr H_\kappa)$ of norm at most one such that \eqref{inter-new} holds.
We say that $\kappa$ has {\it scalar Pick property} if $\kappa$ has the $n$-point scalar Pick property for every positive integer $n$ (this is equivalent to \cite[Definition~9.22]{Mc-0}; see \cite[p.~112]{A-M-0} for a clarification).
In case $\kappa$ is the {\it Szegö kernel of the unit disc $\mathbb D,$} that is, $\kappa(z, w)=\frac{1}{1-z\overline{w}},$ $z, w \in \mathbb D,$ then $\kappa$ has scalar Pick property (see \cite[Theorem~1.3]{A-M}). 
 
In this paper, we address a Nevanlinna-Pick interpolation problem in the right half-plane $\mathbb H_{\delta}:=\{\sigma + it \in \mathbb C : \sigma > \delta\},$ where $\delta$ is a real number. 
Recall that the {\it Riemann zeta function} $\zeta$ is given by $\zeta(s)=\sum_{n=1}^\infty \frac{1}{n^s},$ $s \in \mathbb H_1.$
The {\it Szegö-Dirichlet kernel of the right half-plane $\mathbb H_{1/2}$} is the positive definite kernel $\varkappa : \mathbb H_{1/2} \times \mathbb H_{1/2} \rar \mathbb C$ given by 
\beqn 
\varkappa(s, u) = \zeta(s+\overline{u}), \quad s, u \in \mathbb H_{1/2}.
\eeqn
Let $\mathscr H^2$ denote the reproducing kernel Hilbert space associated with $\varkappa.$ 
The notation $\mathcal D$ is reserved for the set of all functions representable by a convergent Dirichlet
series in some right half-plane.

The starting point of our investigations is the following unpublished fact from \cite[p.~108]{Mc-notes}.
\begin{theorem}[McCarthy] \label{McC}
The Szegö-Dirichlet kernel of $\mathbb H_{1/2}$ does not have the scalar Pick property.
\end{theorem}
The proof of this theorem, as presented in \cite{Mc-notes}, relies on a couple of nontrivial facts related to bounded interpolating sequences (see \cite[Theorem~ 9.10]{Mc-notes} and \cite[Theorem~2.1]{OS}).  
Our first main result of this paper generalizes Theorem~\ref{McC}.  
\begin{theorem} \label{cor-fail-2-point-Pick} Let $m$ be a positive integer and let $\varkappa$ be the Szegö-Dirichlet kernel of the right half-plane $\mathbb H_{1/2}.$ Then 
$\varkappa^m$ does not have the $2$-point scalar Pick property.
\end{theorem}

In spite of the above failure, the right half-plane $\mathbb H_0$ has a network realization formula. Before we state this, let us fix some notations and recall a few facts. Recall that by \cite[Example~$11.4.1$]{Ap},  
\beq \label{recipro-zeta}
\frac{1}{\zeta(s)} = \sum_{n=1}^\infty \mu(n) n^{-s}, \quad s \in \mathbb H_1,
\eeq
where $\mu$ is the {\it M$\ddot{\mathrm{o}}$bius function} given by
\beqn
\mu(n) = 
\begin{cases}
	1 & \text{if $n = 1$},\\ 
	(-1)^j & \text{if $n$ is a product of $j$ distinct primes},\\
	0  & \text{otherwise}.
\end{cases}
\eeqn
Define the kernel $\kappa_\mu : \mathbb H_{1/2} \times \mathbb H_{1/2} \rar \mathbb C$ by  
\beq
\label{kappa-b}
\kappa_{\mu}(s, u) = \sum_{m=1}^\infty (1+\mu(m)) m^{-s-\overline{u}}, \quad s, u \in \mathbb H_{1/2},
\eeq
and note that $\kappa_\mu$ is a positive semi-definite kernel. 
By abuse of notations, we use the symbol $\kappa_{\mu, u}$ for $\kappa_\mu(\cdot, u),$ $u \in \mathbb H_{1/2},$ 
and $\mathscr H_\mu$ for the reproducing kernel Hilbert space $\mathscr H_{\kappa_\mu}$ associated with $\kappa_\mu.$

\begin{theorem}[Network realization formula]
	\label{de_branges_ker}
Let $\varkappa$ denote the Szegö kernel of the right half-plane $\mathbb H_{1/2}$ and $\kappa_{\mu}$ be given by \eqref{kappa-b}.
	For a function $\varphi: \mathbb H_0 \rar \mathbb C,$ the following statements are equivalent$:$
	\begin{itemize}
		\item [$\mathrm{(i)}$] $\varphi$ is in $H^\infty_1(\mathbb H_0) \cap \mathcal D,$
		\item [$\mathrm{(ii)}$] there exist a Hilbert space $\mathscr K,$ a holomorphic function $\psi : \mathbb H_{1/2} \rar \mathscr K$ and a contraction 
		\beqn
		V = \begin{bmatrix} a & 1 \otimes \beta \\ \gamma \otimes 1 & D \end{bmatrix} : \mathbb{C} \oplus \mathscr H^2 \otimes \mathscr K \rar \mathbb{C} \oplus \mathscr H_{\mu} \otimes \mathscr K
		\eeqn
		such that $V$ is isometry on $\bigvee \{1 \oplus  \varkappa_{\overline{s}} \otimes \psi(s) : s \in \mathbb H_{1/2}\}$, where $a \in \mathbb C,$ $\beta \in \mathscr H^2 \otimes \mathscr K$, $\gamma \in \mathscr H_{\mu} \otimes \mathscr K$ and $D : \mathscr H^2 \otimes \mathscr K \rar \mathscr H_{\mu} \otimes \mathscr K$ is a bounded linear transformation satisfying  
		\beq \label{D-contraction}
		D(\varkappa_{\overline{s}} \otimes \psi(s)) = \kappa_{\mu, \overline{s}} \otimes \psi(s) - \gamma, \quad s \in \mathbb H_{1/2},
\\
		\label{phi-expression}
		\varphi(s) = a + \langle  (T_{\overline{s}} \otimes I - D)^{-1}\gamma, \beta \rangle_{_{\mathscr H^2 \otimes \mathscr K}}, \quad s \in \mathbb H_{1/2}
		\eeq
for some bounded linear transformation $T_{\overline{s}}: \mathscr H^2 \rar \mathscr H_{\mu}$ satisfying $T_{\overline{s}}(\varkappa_{\overline{s}})=\kappa_{\mu, \overline{s}},$ $s \in \mathbb H_{1/2}.$
	\end{itemize}
\end{theorem}
In Section~\ref{S2}, we present a proof of Theorem~\ref{cor-fail-2-point-Pick}. This relies on a variant of Pick's lemma for the right half-plane (see Theorem~\ref{pick-theorem-char}).
In Section~\ref{S3}, we prove Theorem~\ref{de_branges_ker}. 
This proof relies on a circle of ideas presented in \cite[Chapter~2]{AMY} (cf. \cite[Theorem 3.4]{Mc-0}). 

\section{Pick's lemma for the right half-plane}\label{S2}

Our proof of Theorem~\ref{cor-fail-2-point-Pick} relies on a variant of Pick's lemma. 
\begin{theorem}\label{pick-theorem-char}
	Let $\lambda_1, \ldots, \lambda_n$ be distinct points in $\mathbb H_{1/2}$ and $w_1, \ldots, w_n \in \mathbb D.$ Suppose that there exists
	$\varphi \in H^{\infty}(\mathbb H_{0}) \cap \mathcal D$ such that $\|\varphi\|_{\infty} \Le 1$ and $\varphi(\lambda_j)=w_j,~j=1, \ldots, n$. Then, we have the following$:$
	\begin{itemize}
		\item [$\mathrm{(i)}$] $[ (1-\overline{w}_iw_j)\zeta(\overline{\lambda}_i+\lambda_j)]_{i,j=1}^n$ is positive semi-definite,  
		\item [$\mathrm{(ii)}$] $\Big[ \frac{1-\overline{w}_iw_j}{
			\overline{\lambda_i}+ \lambda_j}\Big]_{i,j=1}^n$ is positive semi-definite with rank equal to $n.$
	\end{itemize}
\end{theorem}
Recall that the {\it Cayley transform} $C : \mathbb H_0 \rar \mathbb D$ is given by 
\beq
\label{psi-r}
C(z) = \frac{z-1}{z+1}, \quad z \in \mathbb H_0.
\eeq
In the proof of Theorem~\ref{pick-theorem-char}, we require the following preparatory lemma. 
\begin{lemma}
	\label{pick-char-riemann}
	Let $\lambda_1, \ldots, \lambda_n$ be distinct points in $\mathbb H_{0}$ and $w_1, \ldots, w_n \in \mathbb D.$ 
	Then, the following statements are equivalent$:$
	\begin{itemize}
		\item [$\mathrm{(i)}$] 
		the Pick matrix 
		$P:=\Big[\frac{1-\overline{w}_iw_j}{
			\overline{\lambda}_i+ \lambda_j}\Big]_{i,j=1}^n$ is positive semi-definite,   
		\item[$\mathrm{(ii)}$] there exists a function $\phi \in H_1^\infty(\mathbb D)$ such that 
		\beq
		\label{interpolation-valid}
		\phi(C(\lambda_i)) = w_i, \quad i=1, \ldots, n.
		\eeq 
	\end{itemize}
Moreover, the Pick matrix $P$ has rank less than $n$ if and only $\phi$ satisfying \eqref{interpolation-valid} is unique. In this case, $\phi$ is a Blaschke product of degree less than $n.$
\end{lemma}
\begin{proof}
Let 
$Q: = \big[\frac{1-\overline{w}_iw_j}{1- \overline{C(\lambda_i)}C(\lambda_j)}\big]_{i,j=1}^n$ and $R := \big[\frac{(\overline{\lambda}_i+1)(\lambda_j +1)}{2}\big]_{i,j=1}^n.$ 
A routine calculation using \eqref{psi-r} 
shows that $Q=P \diamond R,$ 
where $\diamond$ denotes the Schur product. Since
$R$ is a rank one positive semi-definite matrix with invertible entries, by the Schur product theorem (see \cite[Theorem 10.18]{AMY}), 
$P \succeq 0$ if and only $Q \succeq 0.$ 
Also, by Ballantine's theorem (see \cite[Theorem~3.2]{Sty}), $\text{rank}(Q) \leq \text{rank}(P)\, \text{rank}(R)$. However, since $\text{rank}(R)=1,$ $\text{rank}(Q) \leq \text{rank}(P).$ A similar argument with $R$ replaced by $\big[\frac{2}{(\overline{\lambda}_i+1)(\lambda_j +1)}\big]_{i,j=1}^n,$ shows that $\text{rank}(P) \leq \text{rank}(Q).$   
The equivalence of (i) and (ii) as well as the remaining part now follow from \cite[Theorem~1.3]{A-M}.
\end{proof}

For a positive integer $n,$ let $p_n$ denote the $n$-th prime in an increasing enumeration of prime numbers. Set $\mathcal P_n := \{2^{m_1}3^{m_2} \cdots p_n^{m_n}: m_1, \ldots, m_n \Ge 0\},$ $n \Ge 1.$ 

The following result provides a sufficient condition ensuring that a diagonal kernel does not have the $2$-point scalar Pick property.
\begin{proposition}\label{fail-2-point-Pick}
Let $\mf a = \{a_n\}_{n=1}^\infty$ be a sequence of positive real numbers such that for every $\sigma > 0,$
\beqn
\sum_{j \in \mathcal P_n} a_j j^{-\sigma} < \infty, \quad  n \Ge 1. 	
\eeqn
Let $\kappa_\mf a:\mathbb H_{1/2} \times \mathbb H_{1/2} \rar\ \mathbb C$ be a positive semi-definite kernel given by
\beqn
\kappa_{\mf a}(s, u) = \sum_{m=1}^\infty a_m m^{-s-\overline{u}}, \quad s, u \in \mathbb H_{1/2}.
\eeqn
Assume that there exist distinct points $\lambda_1, \lambda_2$ in $\mathbb H_{1/2}$ and $w_1, w_2 \in \mathbb D$ such that  
	\beq
\label{check}
	\text{$[(1-\overline{w}_iw_j) \kappa_\mf a(\lambda_j, \lambda_i)]_{i,j=1}^2 \succeq 0$ and $\Big[\frac{1-\overline{w}_iw_j}{
			\overline{\lambda_i}+ \lambda_j}\Big]_{i,j=1}^2 \nsucceq 0$}.
	\eeq 
Then $\kappa_\mf a$ does not have the $2$-point scalar Pick property.
\end{proposition}
\begin{proof}
	By Lemma~\ref{pick-char-riemann}, there is no function $h$ in $H_1^\infty(\mathbb D)$ such that $h(C(\lambda_i)) = w_i,~i=1, 2.$ Since 
	\beqn
	H_1^\infty(\mathbb H_0) = \{f\circ C : f \in H_1^\infty(\mathbb D)\},
	\eeqn
	we infer that there does not exist a function $g$ in $H_1^\infty(\mathbb H_{0})$ such that $g(\lambda_i) = w_i,~i=1, 2$. Hence, there is no Dirichlet series in $H_1^\infty(\mathbb H_0)$ such that $g(\lambda_i) = w_i,~i=1, 2.$ Since the multiplier algebra of $\mathscr H(\kappa_{\mf a})$ is contractively contained in $H^\infty(\mathbb H_0) \cap \mathcal D$ (see \cite[Theorem~3.1(5)]{St}), we conclude that $\kappa_\mf a$ does not have $2$-point scalar Pick property.
\end{proof}

Let us recall approximate values of $\zeta$ at certain points which are needed later on (refer to \cite[p.~811]{A-S}):
	\beq \label{zeta-values}
&	\zeta(2) \approx 1.6449, \, \zeta(3) \approx 1.2020, \, \zeta(4) \approx 1.0823, \notag\\ 
&	\zeta(7) \approx 1.0083, \, \zeta(12) \approx 1.0002.
\eeq
In particular, we have
\beq \label{values-zeta}
\frac{\zeta(3)^2}{\zeta(2)\zeta(4)} < \frac{8}{9}.
\eeq 
As a consequence of Proposition~\ref{fail-2-point-Pick}, we prove that the Szegö-Dirichlet kernel of the right half-plane $\mathbb H_{1/2}$ does not have the $2$-point scalar Pick property. 
\begin{proof}[Proof of Theorem~\ref{cor-fail-2-point-Pick}]
Fix a positive integer $m.$ 
By \cite[Equation~(1.2.2)]{Ti}, we may write $\varkappa^m$ as 
\beqn
\varkappa^m(s,u) = \sum_{n=1}^\infty a_m(n)n^{-s-\overline{u}}, \quad s, u \in \mathbb H_{1/2}
\eeqn
for some scalars $a_m(n),$ 
then by \cite[Lemma~3.2]{St},
\beqn
\sum_{j \in \mathcal P_n} a_m(j)j^{-\sigma} = \Big(\sum_{j \in \mathcal P_n} j^{-\sigma}\Big)^m < \infty, \quad  n \Ge 1
\eeqn
for every $\sigma > 0.$ Hence, in view of Proposition~\ref{fail-2-point-Pick}, it now suffices to check that \eqref{check} holds for some choices of distinct points $\lambda_1, \lambda_2$ in $\mathbb H_{1/2}$ and $w_1, w_2 \in \mathbb D$.
To see that, let $\lambda_1 = 1, \lambda_2 = 2, w_1 = 0$ and $w_2 \in \mathbb C$ such that 
\beq
\label{choice-w2}
1/3 < |w_2| < \sqrt{1- \frac{\zeta(3)^2}{\zeta(2)\zeta(4)}}
\eeq (such a $w_2$ exists in view of \eqref{values-zeta}). Note that
\beqn
\Big[(1-\overline{w_i}w_j)\varkappa^m(\lambda_i,\lambda_j)\Big]_{i,j=1}^2 = \begin{pmatrix}
	\zeta(2)^m  & \zeta(3)^m \\
	\zeta(3)^m  & (1-|w_2|^2)\zeta(4)^m
\end{pmatrix},  
\eeqn
which is positive semi-definite as its diagonal entries are positive and its determinant is equal to
\beqn
\zeta(2)^m(1-|w_2|^2)\zeta(4)^m - \zeta(3)^{2m}  \Ge \zeta(3)^{2m}\Big(\frac{\zeta(2)\zeta(4)}{\zeta(3)^2}(1-|w_2|^2)-1\Big) > 0
\eeqn
(in view of \eqref{values-zeta} and \eqref{choice-w2}).
Furthermore, since $|w_2| > 1/3,$ the matrix 
\beqn
\Big[\frac{1-\overline{w}_iw_j}{
			\overline{\lambda_i}+ \lambda_j}\Big]_{i,j=1}^2 = \begin{pmatrix}
		\frac{1}{2}  & \frac{1}{3} \\
		 \frac{1}{3} & \frac{1-|w_2|^2}{4} 
	\end{pmatrix}  
\eeqn
has determinant less than $0.$ Thus \eqref{check} holds, and hence $\varkappa^m$ does not have the $2$-point scalar Pick property.
\end{proof}

\begin{proof}[Proof of Theorem~\ref{pick-theorem-char}]
	Suppose that there exists
	$\varphi$ in $H_1^{\infty}(\mathbb H_{0}) \cap \mathcal D$ such that $\varphi(\lambda_j)=w_j,~j=1, \ldots, n.$ 
However, by \cite[Theorem~3.1]{HLS}, $\mathcal M(\mathscr H^2) = H^\infty(\mathbb H_0) \cap \mathcal D$ (isometrically), and hence $\varphi$ is in the closed unit ball of  $\mathcal M(\mathscr H^2).$ Hence, by \cite[Theorem~5.2]{A-M}, the condition (i) follows. Moreover, the function $h: =\varphi\circ C^{-1}$ belongs to $H_1^\infty(\mathbb D)$ and satisfies $h(C(\lambda_i)) = w_i$ for all $i =1, 2, \ldots, n.$ 
One may now apply Lemma~\ref{pick-char-riemann} to obtain $\Big[ \frac{1-\overline{w}_iw_j}{
			\overline{\lambda_i}+ \lambda_j}\Big]_{i,j=1}^n \succeq 0$.
If this matrix has rank less than $n,$ then by Lemma~\ref{pick-char-riemann}, $\varphi = B \circ C$ for some finite Blaschke product $B$ and hence, $B \circ C \in H^\infty(\mathbb H_0) \cap \mathcal D.$ Since $C$ is a biholomorphism and the zero set of $B$ is finite, this contradicts \cite[Corollary 8.4]{Mc-notes}, completing the proof.	
\end{proof}

\begin{remark}
	Note that both conditions $(\mathrm{i})$ and $(\mathrm{ii})$ in Theorem~\ref{pick-theorem-char} are independent. It is evident from the proof of Proposition~\ref{fail-2-point-Pick} that condition $(\mathrm{i})$ does not necessarily imply condition $(\mathrm{ii})$. 
In general, the condition $(\mathrm{ii})$ also does not imply condition $(\mathrm{i})$.  Indeed, consider $\lambda_1 = 1, \lambda_2 = 6, w_1 = 0$ and $w_2 = \frac{1}{\sqrt{2}}.$ Then, 
	\beqn
	\left[ \frac{1-\overline{w}_iw_j}{
		\overline{\lambda_i}+ \lambda_j}\right]_{i,j=1}^2 = \begin{pmatrix}
		\frac{1}{2} & \frac{1}{7} \\[5pt]
		\frac{1}{7} & \frac{1}{24}
	\end{pmatrix} \succeq 0.
	\eeqn
	However, the matrix
$[ (1-\overline{w}_iw_j)\zeta(\overline{\lambda}_i+\lambda_j)]_{i,j=1}^2 = 
	\left(\begin{smallmatrix}
		\zeta(2) & \zeta(7) \\[5pt]
		\zeta(7) & \frac{\zeta(12)}{2}
	\end{smallmatrix}\right)
$
	is not positive semi-definite, since
	$\zeta(2)\zeta(12) <  2\,\zeta(7)^2$ $($see \eqref{zeta-values}$).$
\end{remark}

\section{A network realization formula} \label{S3}

We need the following lemma in a proof of Theorem~\ref{de_branges_ker}. 

\begin{lemma}
	\label{operator-invertible} 
For every $\lambda \in \mathbb H_{1/2}$ and 
for any auxiliary Hilbert space $\mathcal K$ and a contraction $D:\mathscr H^2 \otimes \mathcal K \rar \mathscr H_{\mu} \otimes \mathcal K,$ 
there exists a bounded linear operator $T_{\lambda} : \mathscr H^2 \rar \mathscr H_\mu$ such that  $T_{\lambda}(\varkappa_\lambda)=\kappa_{\mu, \lambda}$ and $T_\lambda \otimes I - D$ is bounded and invertible, where $T_\lambda \otimes I$ denotes the tensor product of $T_\lambda$ and the identity operator $I$ on $\mathcal K.$   
\end{lemma}
\begin{proof}
For $\lambda = \sigma + it \in \mathbb H_{1/2}$, let $\mathbb C\varkappa_{\lambda}$ denote the subspace of $\mathscr H^2$ spanned by the vector $\varkappa_{\lambda}$. Define the linear mapping $V_\lambda : \mathbb C\varkappa_{\lambda} \rar  \mathbb C\kappa_{\mu, \lambda}$ by setting    
$V_\lambda(\varkappa_{\lambda}) = \kappa_{\mu, \lambda}.$ 
Let $U_\lambda: (\mathbb C\varkappa_{\lambda})^\perp \rar  (\mathbb C\kappa_{\mu, \lambda})^{\perp}$
	be
	any unitary (such a unitary exists since $(\mathbb C\varkappa_{\lambda})^\perp$ and $(\mathbb C\kappa_{\mu, \lambda})^{\perp}$ are infinite dimensional separable Hilbert spaces). For fixed $\alpha \in \mathbb C$ with $|\alpha| > 1$, 
define $T_\lambda: = V_\lambda \oplus \alpha {U}_\lambda$.  
	Since $\alpha U_\lambda$ is invertible, $T_\lambda$ is a bounded invertible operator. 
	We now verify that $\|T_\lambda^{-1}\| < 1.$ To see this, 
	let $\tilde{\epsilon} > 0$ be given by
	\beq \label{epsilon-til-form}
	\tilde{\epsilon}^2 = \frac{\zeta(2\sigma)^2 + 1/2}{\zeta(2\sigma)^2 + 1}.
	\eeq 
	Then,
	\beqn
	\|V_\lambda^{-1} (\kappa_{\mu, \lambda})\|^2 = \| \varkappa_{\lambda}\|^2 =  \zeta(2\sigma)
	&<& \frac{1}{\zeta(2\sigma)} \Big(\zeta(2\sigma)^2 + \frac{1}{2}\Big) \\
	&\overset{\eqref{epsilon-til-form}}=&\frac{\tilde{\epsilon}^2}{\zeta(2\sigma)} \Big(\zeta(2\sigma)^2+1\Big)\\
	&\overset{\eqref{recipro-zeta}}=& \tilde{\epsilon}^2 \Big(\sum_{n=1}^\infty (1+\mu(n))n^{-2\sigma}\Big)\\ 
	&=& \tilde{\epsilon}^2\|\kappa_{\mu, \lambda}\|^2,  
	\eeqn
	and hence, $\|V_\lambda^{-1}\| \Le \tilde{\epsilon}.$
    Thus, since $T_\lambda^{-1} = V_\lambda^{-1} \oplus \alpha^{-1}U_\lambda^{-1}$, we get
	\beq 
	\label{norm<1}
	\|T_\lambda^{-1}\| = \sup\{\|V_\lambda^{-1}\|, |\alpha|^{-1}\|\} \Le \sup\{\tilde{\epsilon}, |\alpha|^{-1}\| \} \overset{\eqref{epsilon-til-form}}< 1.
	\eeq
Since $T_\lambda$ is invertible, so is $T_\lambda \otimes I$, and hence
	\beqn
	\|(T_\lambda \otimes I)^{-1} D \| = \|(T_\lambda^{-1} \otimes I) D \| \Le \|T_\lambda^{-1}\| \|D\| \overset{\eqref{norm<1}}< 1.
	\eeqn
	Therefore,
	\beqn
	T_\lambda \otimes I - D = (T_\lambda \otimes I) (I - (T_\lambda^{-1} \otimes I) D)
	\eeqn
	is bounded and invertible. 	This completes the proof.
\end{proof}
\begin{remark}
Although the line preceding \eqref{norm<1} provides a slightly sharper estimate for $\|V_\lambda^{-1}\|,$ 
in response to the referee's suggestion, we note that the estimate $\|V_\lambda^{-1}\| < 1,$ as required in the above proof, can be directly deduced from the following$:$
\beqn
	\| \varkappa_{\lambda}\|^2 =  \varkappa(\lambda, \lambda)=\zeta(2\sigma) < \zeta(2\sigma)+\frac{1}{\zeta(2\sigma)}=\kappa_{\mu}(\lambda, \lambda)=\|\kappa_{\mu, \lambda}\|^2. 
\eeqn
\end{remark}

Given a Dirichlet series $\varphi \in H^{\infty}(\mathbb H_{0})$ of norm at most $1,$ the kernel
\beqn
\kappa_{\varphi}(s, u): = (1 - \varphi(s)\overline{\varphi(u)})\varkappa(s, u), \quad s, u \in \mathbb H_{1/2}
\eeqn
is positive semi-definite. Since the operator $\mathscr M_{\varphi}$ of multiplication by $\varphi$ defines a bounded linear operator on $\mathscr H^2$ such that $\|M_{\varphi}\| = \|\varphi\|_{\infty}$ (see \cite[Theorem~3.1]{HLS}), this may be deduced from \cite[Theorem 5.21]{P-R}. 

\begin{proof}[Proof of Theorem~\ref{de_branges_ker}] 
$(\text{i})\Rightarrow(\text{ii}):$
Suppose that $\varphi \in H_1^{\infty}(\mathbb H_{0}) \cap \mathcal D.$ Since $\kappa_\varphi$ is positive semi-definite, by \eqref{recipro-zeta}, 
	\beqn 
	1 - \varphi(s)\overline{\varphi(u)} = \kappa_{\varphi}(s, u)\sum_{n = 1}^{\infty} 
	\mu(n) n^{-s-\overline{u}}, \quad s, u \in \mathbb H_{1/2}.
	\eeqn
Thus, for $s, u \in \mathbb H_{1/2},$
	\beq
	\label{implication-model} 
	1 +  \zeta(s+\overline{u})\kappa_{\varphi}(s, u) = \varphi(s)\overline{\varphi(u)} + \kappa_{\mu}(s, u)  \kappa_{\varphi}(s, u). 
	\eeq
Since $\kappa_\varphi$ is positive semi-definite, by Moore's theorem, 
\beqn
\kappa_{\varphi}(s, u)=\inp{\psi(s)}{\psi(u)}_{_{\mathscr K}}, \quad s, u \in \mathbb H_{1/2}
\eeqn
	for a Hilbert space $\mathscr K=\bigvee \{\psi(u) : u \in \mathbb H_{1/2}\}$ and a holomorphic function $\psi : \mathbb H_{1/2} \rar \mathscr K.$
	It now follows from \eqref{implication-model} that 
	\beqn
	&&1 + \langle  \varkappa_{\overline{s}} \otimes \psi{(s)} ,  \varkappa_{\overline{u}} \otimes \psi{(u)} \rangle_{\mathscr H^2 \otimes \mathscr K}\\
	=  &&\varphi(s)\overline{\varphi(u)} + \langle \kappa_{\mu, \overline{s}} \otimes \psi(s), \kappa_{\mu, \overline{u}} \otimes \psi(u) \rangle_{{\mathscr H_{\mu}} \otimes \mathscr K},
	\eeqn
	which may be rewritten as 
	\beqn
	&&	\left\langle \begin{pmatrix} 1 \\ \varkappa_{\overline{s}} \otimes \psi(s) \end{pmatrix},  \begin{pmatrix} 1 \\  \varkappa_{\overline{u}} \otimes \psi(u) \end{pmatrix}\right\rangle_{\mathbb{C} \oplus \mathscr H^2 \otimes \mathscr K} \\ &=& \left \langle \begin{pmatrix} {\varphi(s)} \\  \kappa_{\mu, \overline{s}} \otimes \psi(s) \end{pmatrix},  \begin{pmatrix} {\varphi(u)} \\  \kappa_{\mu, \overline{u}} \otimes \psi(u) \end{pmatrix}\right\rangle_{\mathbb{C} \oplus \mathscr H_{\mu} \otimes \mathscr K}, \quad s, u \in \mathbb H_{1/2}.
	\eeqn
	Therefore, by Lurking isometry lemma (see \cite[Lemma~2.18]{AMY}), there exists a linear isometry $V$ from the linear span of $\{1 \oplus  \varkappa_{\overline{s}} \otimes \psi(s): s \in \mathbb H_{\frac{1}{2}}\}$ into $\mathbb{C} \oplus \mathscr H_{\mu} \otimes \mathscr K$
	such that
	\beq \label{def-V}
	V \begin{pmatrix} 1 \\  \varkappa_{\overline{s}} \otimes \psi(s) \end{pmatrix} = \begin{pmatrix} {\varphi(s)} \\  \kappa_{\mu, \overline{s}} \otimes \psi(s) \end{pmatrix}, \quad s \in \mathbb H_{\frac{1}{2}}.
	\eeq
	Extending $V$ trivially on
	\beqn
	\Big(\bigvee \{1 \oplus  \varkappa_{\overline{s}} \otimes \psi(s) : s \in \mathbb H_{\frac{1}{2}}\}\Big)^{\perp}
	\eeqn
	yields a partial isometry $\tilde{V} : \mathbb{C} \oplus \mathscr H^2 \otimes \mathscr K \rar \mathbb{C} \oplus \mathscr H_{\mu} \otimes \mathscr K.$ If we represent $\tilde{V}$ as a $2\times2$ block matrix, then
	\beqn
	\tilde{V} = \begin{bmatrix} a & 1 \otimes \beta \\ \gamma \otimes 1 & D \end{bmatrix} : \mathbb{C} \oplus \mathscr H^2 \otimes \mathscr K \rar \mathbb{C} \oplus \mathscr H_{\mu} \otimes \mathscr K,
	\eeqn
	where $a \in \mathbb C,$ $\beta \in \mathscr H^2 \otimes \mathscr K,$  $\gamma \in  \mathscr H_{\mu} \otimes \mathscr K$ and $D : \mathscr H^2 \otimes \mathscr K \rar \mathscr H_{\mu} \otimes \mathscr K$ is a bounded linear transformation. This together with \eqref{def-V} yields 
	\beq
	\label{realization-matrix-action}
	\begin{pmatrix} {\varphi(s)} \\  \kappa_{\mu, \overline{s}} \otimes \psi(s) \end{pmatrix} &=& 	\tilde{V}\begin{pmatrix} 1 \\  \varkappa_{\overline{s}} \otimes \psi(s) \end{pmatrix}  \notag\\
	&=&  \begin{bmatrix} a & 1 \otimes \beta \\ \gamma \otimes 1 & D \end{bmatrix}\begin{pmatrix} 1 \\  \kappa_{\overline{s}} \otimes \psi(s) \end{pmatrix} \notag\\
	&=& \begin{pmatrix} a + \langle  \varkappa_{\overline{s}} \otimes \psi(s), \beta  \rangle_{_{\mathscr H^2 \otimes \mathscr K}} \\  \gamma +  D ( \varkappa_{\overline{s}} \otimes \psi(s)) \end{pmatrix}, \quad s \in \mathbb H_{1/2}.
	\eeq
By Lemma~\ref{operator-invertible}, for every $s \in \mathbb H_{1/2},$ there exists a bounded linear operator $T_{s} : \mathscr H^2 \rar \mathscr H_{\mu}$ such that  $T_{s}(\varkappa_s)=\kappa_{\mu, s}$ and $T_s \otimes I - D$ is invertible.  
It now follows from \eqref{realization-matrix-action} that 
	\beqn
	(T_{\overline{s}} \otimes I - D)(\varkappa_{\overline{s}} \otimes \psi(s)) = \gamma, 
	\eeqn
and hence 
	\beqn
	\varkappa_{\overline{s}} \otimes \psi(s) = (T_{\overline{s}} \otimes I - D)^{-1}\gamma, \quad  s \in \mathbb H_{1/2}.
	\eeqn
	This combined with \eqref{realization-matrix-action} gives the desired formula for $\varphi$ in \eqref{phi-expression}.
	
	$(\text{ii})\Rightarrow(\text{i}):$ 
	By \eqref{D-contraction},
	\beq
	\label{gamma-expression}
	\gamma = \kappa_{\mu, \overline{s}} \otimes \psi(s) - D( \varkappa_{\overline{s}} \otimes \psi(s)) = (T_{\overline{s}} \otimes I - D) (\varkappa_{\overline{s}} \otimes \psi(s)), \quad s \in \mathbb H_{1/2}.
	\eeq
Since $T_{\overline{s}} \otimes I - D$ is invertible, 
	\beqn
	\varkappa_{\overline{s}} \otimes \psi(s) = (T_{\overline{s}} \otimes I - D)^{-1}\gamma, \quad  s \in \mathbb H_{1/2}.
	\eeqn
It follows that	
\beqn
	V\begin{pmatrix} 1 \\  \varkappa_{\overline{s}} \otimes \psi(s)
	\end{pmatrix} 
	&=& \begin{pmatrix} a + \langle  \varkappa_{\overline{s}} \otimes \psi(s), \beta  \rangle_{_{\mathscr H^2 \otimes \mathscr K}}
		\\  \gamma +  D(\varkappa_{\overline{s}} \otimes \psi(s))
	\end{pmatrix}\\
	&=& \begin{pmatrix} a + \langle  (T_{\overline{s}} \otimes I - D)^{-1}\gamma, \beta  \rangle_{_{\mathscr H^2 \otimes \mathscr K}} \\  \gamma +  D(\varkappa_{\overline{s}} \otimes \psi(s)) \end{pmatrix}\\
	&\overset{\eqref{phi-expression}\& \eqref{gamma-expression}}=& \begin{pmatrix} \varphi(s) \\  \kappa_{\mu, \overline{s}} \otimes \psi(s) \end{pmatrix}, \quad s \in \mathbb H_{1/2}.
	\eeqn
	Since $V$ is an isometry on $\bigvee \{1 \oplus  \varkappa_{\overline{s}} \otimes \psi(s) : s \in \mathbb H_{\frac{1}{2}}\}$, we have
	\beqn
	&&1 + \langle  \varkappa_{\overline{s}} \otimes \psi(s),  \varkappa_{\overline{u}} \otimes \psi(u) \rangle_{\mathscr H^2 \otimes \mathscr K}\\
	=  &&\varphi(s)\overline{\varphi(u)} + \langle \kappa_{\mu, \overline{s}} \otimes \psi(s), \kappa_{\mu, \overline{u}} \otimes \psi(u) \rangle_{{\mathscr H_{\mu}} \otimes \mathscr K}, \quad s, u \in \mathbb H_{1/2}.
	\eeqn
	By the reproducing property and \eqref{kappa-b}, we get
	\beqn
	1 = \varphi(s)\overline{\varphi(u)} + \langle \psi(s), \psi(u) \rangle_{\mathscr K} \sum_{m=1}^\infty \mu(m)m^{-s-\overline{u}}, \quad s, u \in \mathbb H_{1/2},
	\eeqn
	and hence, by \eqref{recipro-zeta}, 
	\beqn
	(1- \varphi(s)\overline{\varphi(u)})\zeta(s+\overline{u}) = \langle \psi(s), \psi(u) \rangle_{\mathscr K}. 
	\eeqn
	Equivalently, $(1- \varphi(s)\overline{\varphi(u)})\zeta(s+\overline{u})$ is a positive semi-definite kernel on $\mathbb H_{1/2}.$ Hence, by \cite[Theorem~5.21]{P-R}, $\varphi \in \mathcal M(\mathscr H^2)$ with multiplier norm at most $1.$ Since $\mathcal M(\mathscr H^2) = H^\infty(\mathbb H_0) \cap \mathcal D$ (isometrically) (see \cite[Theorem~3.1]{HLS}), it follows that $\varphi \in H^\infty(\mathbb H_0) \cap \mathcal D$ with $\|\varphi\|_\infty \Le 1.$ This completes the proof. 
\end{proof}


\section*{Concluding remarks}

We conclude this paper with an observation towards a possible converse of Theorem~\ref{pick-theorem-char}. Suppose that Theorem~\ref{pick-theorem-char}(ii) holds. Then, by Lemma~\ref{pick-char-riemann}, 
the interpolation problem 
\beqn
\varphi(\lambda_i) = w_i, \quad i=1,2, \ldots, n~\text{with}~\varphi \in H_1^\infty(\mathbb H_0)
\eeqn
admits more than one solutions. Hence, by an application of \cite[Corollary~8.82]{A-M}, the set of all solutions $\varphi$ is of the form
\beqn
\mathcal S:=\{\tilde{g}_{1,1} + \tilde{g}_{1,2} \tilde{g}_{2,1} (1-\tilde{g}_{2,2} h)^{-1} : h \in H_1^\infty(\mathbb H_0) \}, 
\eeqn 
where $\tilde{g}_{i,j} = g_{i,j} \circ C$ (see \eqref{psi-r}), each $g_{i,j}$ is a rational function of degree at most $n$ with poles outside the closed unit disc, $\tilde{g}_{1,2}(\lambda_i) = 0$ for all $i =1,2, \ldots, n,$ and the matrix $[g_{i,j}]_{i,j=1}^2$ is unitary on the unit circle $\mathbb T.$ 
However, even if in addition Theorem~\ref{pick-theorem-char}(i) holds, we do not know whether there exists a solution $\varphi \in \mathcal S \cap \mathcal D.$ 

\vskip.3cm
\noindent 
\textit{Acknowledments}: The authors would like to thank Prof. B. K. Das for some helpful discussions on the subject matter of this paper. The second author is supported by the Institute post-doctoral fellowship, IIT Bombay. Finally, we thank the anonymous referee for several suggestions that improved the original presentation.

\vskip.2cm
\noindent 
\textit{Statement and Declarations:}

\vskip.3cm

\noindent
{\bf Conflict of interest} The authors declare that they have no conflict of interest.

\vskip.3cm

\noindent 
{\bf Data Availability} No data was used for the research described in the article.

\vskip.3cm

\noindent 
{\bf Funding} This research received no specific grant from any funding agency.

\vskip.3cm

\noindent 
{\bf Author contributions} Both the authors equally contributed to this paper.

\end{document}